\documentclass{amsart}
\usepackage{graphicx} 
\usepackage[utf8]{inputenc}
\usepackage{amsmath}
\usepackage{amssymb, amsfonts}
\usepackage{amsthm,mathtools,comment}
\usepackage{amsmath,amsfonts,amssymb,mathrsfs,amsthm,bbm}
\usepackage{tikz}
\usepackage{tikz-cd}
\usepackage{mathrsfs}
\usepackage{hyperref}
\usepackage{cleveref}
\usepackage{mathscinet}
\usepackage{enumitem}
\theoremstyle{plain}
\newtheorem{theorem}{Theorem}[section]
\newtheorem{lemma}[theorem]{Lemma}

\newtheorem{corollary}[theorem]{Corollary}

\newtheorem{conjecture}[theorem]{Conjecture}
\newtheorem*{theorem*}{Theorem}

\theoremstyle{definition}
\newtheorem{definition}[theorem]{Definition}
\newtheorem{example}[theorem]{Example}
\newtheorem{remark}[theorem]{Remark}

\theoremstyle{remark}

\usepackage{mathrsfs}
\usepackage{hyperref}
\usepackage{cleveref}
\usepackage{todonotes,mathscinet}

\newcommand{\op}{\operatorname}

\newcount\cols
{\catcode`,=\active\catcode`|=\active
 \gdef\Young#1{\hbox{$\vcenter
 {\mathcode`,="8000\mathcode`|="8000
  \def,{\global\advance\cols by 1 &}%
  \def|{\cr
        \multispan{\the\cols}\hrulefill\cr
        &\global\cols=2 }%
  \offinterlineskip\everycr{}\tabskip=0pt
  \dimen0=\ht\strutbox \advance\dimen0 by \dp\strutbox
  \halign
   {\vrule height \ht\strutbox depth \dp\strutbox##
    &&\hbox to \dimen0{\hss$##$\hss}\vrule\cr
    \noalign{\hrule}&\global\cols=2 #1\crcr
    \multispan{\the\cols}\hrulefill\cr%
   }
 }$}}
\gdef\Skew(#1:#2){\hbox{$\vcenter
{\mathcode`,="8000\mathcode`|="8000
  \dimen0=\ht\strutbox \advance\dimen0 by \dp\strutbox
  \def\boxbeg{\vbox
    \bgroup\hrule\kern-0.4pt\hbox to\dimen0\bgroup\strut\vrule\hss$}%
  \def\boxend{$\hss\egroup\hrule\egroup}%
  \def,{\boxend\boxbeg}%
  \def|##1:{\boxend\vrule\egroup\nointerlineskip\kern-0.4pt
    \moveright##1\dimen0\hbox\bgroup\boxbeg}%
  \def\\##1\\##2:{\boxend\vrule\egroup\nointerlineskip\kern-0.4pt
    \kern ##1\dimen0\moveright##2\dimen0\hbox\bgroup\boxbeg}%
  \moveright#1\dimen0\hbox\bgroup\boxbeg#2\boxend\vrule\egroup
 }$}}
}

\title[Unveiling Crystal Embeddings]{Unveiling Crystal Embeddings: New Perspectives on String Polytopes and Atomic Decompositions}

\author{Lara Bossinger, Jacinta Torres}

\begin{document}

\begin{abstract}
We present $n-1$ different  embeddings of string polytopes $\mathcal{S}_{\mathbf{i}}(\lambda) \hookrightarrow \mathcal{S}_{\mathbf{i}}(\lambda + \theta)$ of type $A_{n-1}$, where $\mathbf{i}$ is an arbitrary reduced expression of the longest element of the symmetric group and $\theta$ is the highest root. We characterize their compatibility with the  crystal structure on the string polytopes, and formulate a conjecture describing how to obtain $n-1$ different atomic decompositions of the crystal $\mathcal{B}(k\theta)$, where $\theta$ is the highest root.  
\end{abstract}

\maketitle

\section{Introduction}

String polytopes were introduced by Littelmann \cite{conescrystalspatterns} and Berenstein--Zelevinsky \cite{BerensteinZelevinsky} to parametrize elements in Luztig's dual canonical basis \cite{Lusztig1,Lusztig2}. 
As a generalization of Gelfand--Tsetlin polytopes they are key objects in constructions of toric degenerations of flag and Schubert varieties \cite{caldero:toric,FFL_essential} and a source of deep insights in the interplay of Lie theory and cluster algebras \cite{bea_polyhedral,BossingerFourier}. 
Moreover, they possess a natural crystal structure \cite{Kashiwara:crystal}, which was made precise for type $A_{n-1}$ in \cite{beaglebvolker} using the combinatorics of \textit{wiring diagrams} and Gleizner--Postnikov paths \cite{Gleizer-Postnikov}. 

Atomic decompositions of crystals \cite{LL21} arise in the computation of positive combinatorial formulas for Kostka--Foulkes polynomials, or $q$-analogues of weight multiplicities, which beyond type $A_{n-1}$ remains an open problem. They have been used by Patimo in \cite{Pat} to obtain a geometric proof of Lascoux--Sch\"utzenberger's acclaimed charge formula in type $A_{n-1}$ \cite{lascoux1978conjecture}, suggesting a global approach to obtain positive formulas outside of type $A_{n-1}$. However, even for type $A_{n-1}$ many unanswered questions remain. For instance, there is so far only one known atomic decomposition of type $A_{n-1}$ crystals. 

Given a reduced word $\mathbf{i}$ of the longest element of the Weyl group and a dominant weight $\lambda$, we will denote the corresponding string polytope by $\mathcal{S}_{\mathbf{i}}(\lambda)$. 
In this note we present certain embeddings of string polytopes 

\[\phi_j: \mathcal{S}_{\mathbf{i}}(\lambda) \hookrightarrow \mathcal{S}_{\mathbf{i}}(\lambda + \theta), \]

\noindent
for $j \in [1,n-1]$ where $\lambda$ is a dominant weight of type $A_{n-1}$, $\theta$ is the highest root of the corresponding root system, and $\bf{i}$ is a reduced expression for the longest element  of $S_n$. Moreover, we show in Theorem \ref{thm:epsilon and phi after adding wt zero} and Corollary \ref{thm:crystal1} that for $j = 1, n-1$ these embeddings $\phi_{j}$ commute with the crystal operators, whenever they are defined in the domain. The other operators $\phi_{j}$ also present certain compatibility with the crystal structure - we describe this in detail as well in Corollary \ref{cor:proj}. We conclude the paper with Conjecture \ref{conj:atoms}, describing in it a way to produce atomic decompositions for the string polytope $\mathcal{S}_{\mathbf{i}}(k\theta)$, one for each $j \in [1,n-1]$. This is very desirable, since, so far, there is only one known way, due to Lenart--Lecouvey and Patimo, who present different but equivalent approaches, of producing an atomic decomposition of crystals of type $A_{n-1}$ \cite{Pat, LL21}. We remark that finding atomic decompositions of crystals using such embeddings is in close relation with the computation of pre-canonical bases on affine Hecke algebras  \cite{LPP}.\\

The approach followed in this article is inspired by recent joint work of  Patimo with JT \cite{PT23}, where positive combinatorial formulas for Kostka--Foulkes polynomials of low rank, outside of type $A_{n-1}$,  were found using atomic decompositions of crystals obtained using embeddings very similar to those we present in the present work. This is an unusual instance of a discovery in type $A_{n-1}$ which was inspired by what is happenning for the representation theory of another classical Lie algebra! \\

{\bf Acknowledgements.} 
We are grateful to Bea de Laporte (former Schumann) for inspiring discussions and patiently explaining her results and observations. 
We thank Leonardo Patimo for the insightful discussions about atomic decompositions.
LB is supported by the UNAM DGAPA PAPIIT project IA100724 and the CONAHCyT project
CF-2023-G-106.
JT is supported by the grant SONATA NCN UMO-2021/43/D/ST1/02290 and partially supported by the grant MAESTRO NCN UMO-2019/34/A/ST1/00263.

\section{Crystals and string polytopes}

A crystal for a complex finite dimensional Lie algebra $\mathfrak{g}$ consists of a non-empty set $B$ together with maps 
\begin{align*}
\operatorname{wt}:& B \longrightarrow X \\
e_{i},f_{i}:& B \longrightarrow B \sqcup \left\{ 0\right\}, i \in [1, \operatorname{rank}(\mathfrak{g})]
\end{align*}

\noindent such that for all $b,b' \in B$:

\begin{enumerate}	
\item $b' = e_i(b)$ if and only if $b = f_i(b')$,
\item if $f_i(b) \neq 0 $ then $\textsf{wt}(f_i(b)) = \textsf{wt}(b)-\alpha_i$;
\item
if $e_i(b) \neq 0$, then
$\textsf{wt}(e_i(b)) = \textsf{wt}(b)+\alpha_i$, and
\item  $\varphi_i(b)-\varepsilon_i(b)=  \langle \textsf{wt}(b),\alpha_i^\vee  \rangle$,
\end{enumerate}

\noindent where 
$
\varepsilon_i(b)=\max\{a \in \mathbb{Z}_{\geq 0} :e_i^a(b)\neq 0\} $  and $\varphi_i(b)=\max\{a \in \mathbb{Z}_{\geq 0 }:f_i^a(b)\neq 0\}$, and where $X$ is the lattice of integral weights, $X^{\vee}$ is the lattice of integral coweights, $\left\{ \alpha_{i} \right\}_{i \in [1, \operatorname{rank}(\mathfrak{g})]} \subset X$ is the set of simple roots, $\alpha^{\vee}_{i} \in X^{\vee}$ is the coroot associated to the simple root $\alpha_{i}$, and $ \langle - , - \rangle : X \times X^{\vee} \rightarrow \mathbb{Z}$ is the associated pairing.\\

To each such crystal $B$ is associated a \textit{crystal graph}, a coloured directed graph with vertex set $B$ and edges coloured by elements $i \in [1,\operatorname{rank}(\mathfrak{g})]$, where if $f_{i}(b) = b'$ there is an arrow $b \overset{i}{\rightarrow} b'$. A crystal is irreducible if its corresponding crystal graph is connected and finite. A seminormal crystal is called normal if it is isomorphic to the crystal of a representation of $\mathfrak{g}$. 
Irreducible normal crystals are thus indexed by dominant integral weights of $\mathfrak{g}$. 
We refer the reader to \cite{BSch17} for more background on crystals.

\subsection{Tensor products of crystals}
Given two seminormal crystals $X$ and $Y$, their tensor product $X\otimes Y$ is defined as follows. As a set, it is the cartesian product $X \times Y$. The crystal operators are defined, for each $i \in [1, \operatorname{rank}(\mathfrak{g})]$, $x \in X$ and $b \in B$ as follows: 

  \begin{equation}
    f_{i}(x\otimes y) = 
    \begin{cases}
      f_{i}(x) \otimes y & \text{if} \mbox{  } \varphi_{i}(y)\leq \varepsilon_{i}(x),\\
      x \otimes f_{i}(y)        & \text{if} \mbox{  } \varphi_{i}(y)> \varepsilon_{i}(x).
    \end{cases}
  \end{equation}

\noindent and

  \begin{equation}
    e_{i}(x\otimes y) = 
    \begin{cases}
      e_{i}(x) \otimes y & \text{if} \mbox{  } \varphi_{i}(y)< \varepsilon_{i}(x),\\
      x \otimes f_{i}(y)        & \text{if} \mbox{  } \varphi_{i}(y)\geq \varepsilon_{i}(x).
    \end{cases}
  \end{equation}

\noindent Also, define

\begin{align*}
\varphi_{i}(x\otimes y) &= \operatorname{max}\{\varphi_{i}(x), \varphi_{i}(y) + \langle \operatorname{wt}(y) , \alpha^{\vee}_{i} \rangle\} \\
\varepsilon_{i}(x\otimes y) &= \operatorname{max}\{\varepsilon_{i}(x), \varepsilon_{i}(y) - \langle  \operatorname{wt}(y) , \alpha^{\vee}_{i} \rangle\}.
\end{align*}


An element $b \in B$ is called a \textbf{highest}, respectively \textbf{lowest} weight vertex, if $\varepsilon_{i}(b) = 0 \mbox{  } \forall i \in [1, \operatorname{rank}(\mathfrak{g})]$, respectively $\varphi_{i}(b) = 0 \mbox{  } \forall i \in [1, \operatorname{rank}(\mathfrak{g})]$. The weight of a lowest weight vertex is called a lowest weight; respectively, a highest weight. If the crystal is irreducible, the highest and lowest weight vertices are unique. For a dominant integral weight  $\lambda$ we denote by $\mathcal{B}(\lambda)$ the corresponding  crystal associated to the irreducible representation of $\mathfrak{g}$ 
of highest weight $\lambda$. We denote its highest weight element by $b_{\lambda}$. 

\begin{remark}
\label{rem1}
Our convention for tensor products coincides with \cite{BSch17}; it is in particular opposite to Kashiwara's \cite{kashiwara1990crystalizing, kashiwara1991crystal, kashiwara1994crystal}.
The tensor product of crystals is an associative operation, as is shown in \cite[Proposition 2.32.]{BSch17}.
\end{remark}

\subsection{String cones and polytopes}

In this section we recall a model, introduced by Berenstein--Zelevinsky and Littelmann, for the irreducible highest weight crystals $\mathcal{B}(\lambda)$, which only depends on the choice of a reduced expression $\mathbf{i} = (i_1,\cdots i_N)$ of the longest element $w_0$ of the Weyl group $W$. We follow \cite{conescrystalspatterns, BerensteinZelevinsky}. Now, given a vector $\underline{a} = (a_1, \cdots a_N) \in \mathbb{Z}_{>0}^{N}$, consider 

\[f_{\underline{a}}b_\lambda := f^{a_1}_{i_1}\cdots f^{a_N}_{i_N}(b_{\lambda}),\]

\noindent
where $v_{\lambda} \in \mathcal{B}(\lambda)$ is the highest weight element. 

\begin{definition}
    \label{adaptedstring}
A vector $\underline{a} = (a_1, \cdots a_N) \in \mathbb{Z}_{>0}^{N}$ is called an \textit{adapted string} if $f_{\underline{a}}b_\lambda \neq 0$ and 

\[e_{i_k}(f^{a_{k+1}}_{i_k+1}\cdots f^{a_N}_{i_N}(b_{\lambda})) = 0\]

\noindent for all $1 \leq k \leq N-1$. The cone $\mathcal{S}_{\underline{i}} \subset \mathbb{R}^{N}$ generated by all the adapted strings is called the \textit{string cone} or \textit{cone of adapted strings} associated to $\mathbf{i}$. 
\end{definition}

\begin{theorem}[Berenstein--Zelevinsky \cite{BerensteinZelevinsky}]
The cone $\mathcal{S}_{\underline{i}}$ is a semigroup, that is, $x+y \in \mathcal{S}_{\underline{i}}$ for $x,y \in \mathcal{S}_{\underline{i}}$.
\end{theorem}

\begin{theorem}[Littelmann \cite{conescrystalspatterns}]
Let $\mathbf{i} = (i_1, \cdots, i_N)$ be a reduced expression of the longest Weyl group element $w_0$, i.e. $w_0 = s_{i_1} \cdots s_{i_N}$ is reduced. Let $\lambda \in X^{+}$ be a dominant integral weight and let $\mathcal{S}^{\lambda}_{\mathbf{i}}$ be the polytope defined by the inequalities

\begin{align}
\label{lambdaeqns}
\left\{a_{r} \leq \langle \lambda - \sum _{r < j \leq N} a_{j}\alpha_{i_j}, \alpha^{\vee}_{i_{r}}\rangle:   1 \leq r \leq N \right\}
\end{align}

\noindent
Then the lattice points in $\mathcal{S}^{\lambda}_{\mathbf{i}}$ are the adapted strings for $\lambda$ with respect to the reduced decomposition $\mathbf{i}$. 
\end{theorem}

\section{String polytopes via wiring diagrams}

Throughout the rest of the paper unless stated otherwise we will work with $\mathfrak{g} = \mathfrak{sl}(n,\mathbb{C})$ and fix a reduced expression $\mathbf{i} =(i_{1},..., i_{N})$ of the longest element of the symmetric group $S_{n}$.

\begin{definition}
\label{def:wiringdiagram}
The \textit{wiring diagram} or \textit{pseaudoline arrangement} for $\mathbf{i}$ a diagram consisting of $n$ piecewise straight lines calles \emph{wires} $\ell_i,\hbox{  }  1 \leq i \leq n$, labelled at the left from bottom to top, such that their crossings are dictated by the reduced expression read from left to right. 
This means that the wires are parallel to each other at the left hand side, with the wires at levels $i$ and $i+1$ crossing each other whenever there is an $``i"$ in $\mathbf{i}$. 
See Figure \ref{fig:wiring} for an example.
\end{definition}


\begin{figure}[hbt!]
\begin{tikzpicture}[scale=.75]
    \draw[rounded corners] (-1,0)--(0,0)--(4,4)--(11,4);
    \draw[rounded corners] (-1,1)--(0,1)--(1,0)--(4,0)--(7,3)--(11,3);
    \draw[rounded corners] (-1,2)--(1,2)--(2,1)--(4,1)--(5,0)--(7,0)--(9,2)--(11,2);
    \draw[rounded corners] (-1,3)--(2,3)--(3,2)--(5,2)--(6,1)--(7,1)--(8,0)--(9,0)--(10,1)--(11,1);
    \draw[rounded corners] (-1,4)--(3,4)--(4,3)--(6,3)--(7,2)--(8,2)--(10,0)--(11,0);

    \node[left] at (-1,4) {$\ell_5$};
    \node[left] at (-1,3) {$\ell_4$};
    \node[left] at (-1,2) {$\ell_3$};
    \node[left] at (-1,1) {$\ell_2$};
    \node[left] at (-1,0) {$\ell_1$};

    \node[below] at (.5,-.5) {$1$};
    \node[below] at (1.5,-.5) {$2$};
    \node[below] at (2.5,-.5) {$3$};
    \node[below] at (3.5,-.5) {$4$};
    \node[below] at (4.5,-.5) {$1$};
    \node[below] at (5.5,-.5) {$2$};
    \node[below] at (6.5,-.5) {$3$};
    \node[below] at (7.5,-.5) {$1$};
    \node[below] at (8.5,-.5) {$2$};
    \node[below] at (9.5,-.5) {$1$};
    \draw[->] (-1,0)--(-.5,0);
    \draw[->] (2,0)--(2.5,0);
    \draw[->] (5.5,0)--(6,0);
    \draw[->] (8.75,0)--(8.5,0);
    \draw[<-] (10.5,0)--(10.75,0);

    \draw[->] (.5,.5)--(1,1);
    \draw[->] (4.5,.5)--(5,1);
    \draw[->] (7.5,.5)--(8,1);
    
    \draw[->] (-1,1)--(-.5,1);
    \draw[->] (2.5,1)--(3,1);
    \draw[<-] (6.5,1)--(6.75,1);
    \draw[<-] (10.5,1)--(10.75,1);

    \draw[->] (1.5,1.5)--(2,2);
    \draw[->] (5.5,1.5)--(6,2);

    \draw[->] (0,2)--(.5,2);
    \draw[->] (4.5,2)--(4,2);
    \draw[->] (7.75,2)--(7.5,2);
    \draw[->] (9.5,2)--(10,2);

    \draw[->] (2.5,2.5)--(3,3);

    \draw[<-] (1,3)--(1.5,3);
    \draw[<-] (5,3)--(5.5,3);
    \draw[->] (8.5,3)--(9,3);

    \draw[<-] (2,4)--(2.5,4);
    \draw[->] (7,4)--(7.5,4);
    \draw[thick,rounded corners,blue] (-1,2)--(1,2)--(2,1)--(4,1)--(4.5,.5);
    \draw[thick, blue] (4.5,.5)--(5.5,1.5);
    \draw[thick, blue, rounded corners] (5.5,1.5)--(5,2)--(3,2)--(2,3)--(-1,3);

\end{tikzpicture}
    \label{fig:wiring} 
\caption{The wiring diagram $\op{wd}(\mathbf{i})$ for $n = 5$ and $\mathbf{i} = (1,2,3,4,1,2,3,1,2,1)$ with orientation $(\ell_3,\ell_4)$. In blue, a rigorous path in $\Gamma_3$ determined by its turning points $v_{3,1}$ and $v_{1,4}$.}
\end{figure}

\subsection{Rigorous paths}
\label{diagramnotation}
In this subsection we recall the combinatorics of rigorous paths and how they encode inequalities for string cones following Gleizer and Postnikov \cite{Gleizer-Postnikov} and crystal operators following Genz, Koshevoy and Schumann \cite{beaglebvolker}.

For every pair $(\ell_i,\ell_{i+1})$ with $1\le i\le n-1$ we give an orientation to a wiring diagram by orienting wires $\ell_1,\dots,\ell_i$ from left to right and wires $\ell_{i+1},\dots,\ell_{n}$ from right to left, see Figure~\ref{fig:wiring}. 
We call the orientation of wires  $\ell_1,\dots,\ell_i$ from right to left and wires $\ell_{i+1},\dots,\ell_{n}$ from left to right the \emph{opposite orientation for $(\ell_i,\ell_{i+1})$}. 
Consider an oriented path with three consecutive crossings $v_{k-1}\to v_{k}\to v_{k+1}$ belonging to the same wire $\ell_i$. 
Then $v_k$ is the intersection of $\ell_i$ with some line $\ell_j$, denote this by $v_k=v_{(i,j)}$.
If either $i<j$ and both wires are oriented to the left, or $i>j$ and both wires are oriented to the right, the path is called \emph{non-rigorous}. 
Figure~\ref{fig:rigorous} shows these two situations. A path is called \emph{rigorous} if it is not non-rigorous. 

\begin{definition}\label{def:GPpath} 
Let $\bf{i}$ be a fixed reduced expression of $w_0 \in S_{n}$. 
A \emph{left Gleizer-Postnikov path}, or short \emph{left GP-path}, (respectively, \emph{right Gleizer-Postnikov path}, or short \emph{right GP-path}) is a rigorous path $\gamma$ in $\op{wd}({\bf i})$ endowed with some orientation $(\ell_a,\ell_{a+1})$ (respectively, opposite orientation) for $a\in[n-1]:=\{1,\dots,n-1\}$. 
The set of all such left GP-paths (respectively, right GP-paths) is denoted by $\Gamma_a$ (respectively, $\Gamma_a^*$).
\end{definition}

\begin{figure}[ht]
\centering
\begin{center}
\begin{tikzpicture}[scale=.8]

\draw[rounded corners] (3,0) -- (2,0) -- (1,1) -- (0,1);
        \draw[-<] (3,0) -- (2.5,0);
        \draw[-<] (0.8,1) -- (.5,1);
\draw[rounded corners] (3,1) -- (2,1) -- (1,0) -- (0,0);
    \draw[-<] (3,1) -- (2.5,1);
    \draw[-<] (0.8,0) -- (.5,0);
\draw[<-, ultra thick, red] (1.9,0.9) -- (1.1,0.1);

\begin{scope}[yshift=-2cm]
  \draw[rounded corners] (3,0) -- (2,0) -- (1,1) -- (0,1);
        \draw[-<] (.5,1) -- (.8,1);
        \draw[-<] (2.2,0) -- (2.5,0);
\draw[rounded corners] (3,1) -- (2,1) -- (1,0) -- (0,0);
    \draw[-<] (.5,0) -- (.8,0);
    \draw[-<] (2.2,1)-- (2.5,1);
\draw[->, ultra thick, red] (1.9,0.1) -- (1.1,0.9);
\end{scope}

\begin{scope}[xshift=5cm]
    \draw[rounded corners] (3,0) -- (2,0) -- (1,1) -- (0,1);
        \draw[-<] (3,0) -- (2.5,0);
        \draw[-<] (0.8,1) -- (.5,1);
\draw[rounded corners] (3,1) -- (2,1) -- (1,0) -- (0,0);
    \draw[-<] (3,1) -- (2.5,1);
    \draw[-<] (0.8,0) -- (.5,0);
\draw[<-, ultra thick, red] (1.9,0.1) -- (1.1,0.9);

\begin{scope}[yshift=-2cm]
  \draw[rounded corners] (3,0) -- (2,0) -- (1,1) -- (0,1);
        \draw[-<] (.5,1) -- (.8,1);
        \draw[-<] (2.2,0) -- (2.5,0);
\draw[rounded corners] (3,1) -- (2,1) -- (1,0) -- (0,0);
    \draw[-<] (.5,0) -- (.8,0);
    \draw[-<] (2.2,1)-- (2.5,1);
\draw[->, ultra thick, red] (1.9,0.9) -- (1.1,0.1);
\end{scope}
\end{scope}

\end{tikzpicture}
\end{center}
\caption{On the left (resp. right) the two red arrows are forbidden in left (resp. right) rigorous paths.}\label{fig:rigorous}
\end{figure}
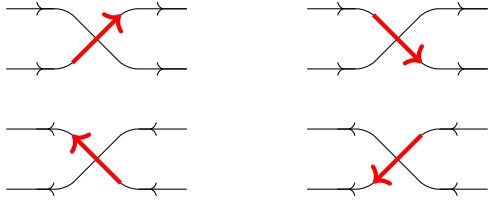

\begin{remark}
    \label{rem:reineke}
Rigurous paths appear as \textit{Reineke crossings} in \cite{beaglebvolker}.
\end{remark}


Rigorous paths allow us to describe the string cones and polytopes comnbinatorially as follows.

\begin{theorem}[\cite{Gleizer-Postnikov} Corollary 5.8, \cite{beaglebvolker} Theorem 6.1]
    Let ${\bf i}$ be a reduced expression of $w_0$. Then the corresponding string cone is given by
    \[
    \mathcal S_{\bf i} = \{x\in \mathbb R^N: \langle x,r(\gamma) \rangle \ge 0 \text{ for all } \gamma \in \Gamma_a^*, a\in [n-1]\}.
    \]
    Let $\lambda=\sum_{a=1}^{n-1}\lambda_a\omega_a \in X^+$. Then the corresponding string polytope is given by
    \[
    \mathcal S_{\bf i}(\lambda) = \{x\in\mathcal S_{\bf i}:  \langle x,r(\gamma') \rangle \le \lambda_a  \text{ for all } \gamma'\in \Gamma_a, a\in [n-1]\}.
    \]
\end{theorem}

\subsection{Crystal structure}
\label{sec:crystalstructure}
The set of lattice points in the string polytope $\mathcal{S}_{\mathbf{i}}(\lambda)$ has an instrinsic crystal structure which can be described as follows. Clearly, for any lattice point $\underline{x} = (x_{i_1},\dots, x_{i_{N}}) \in \mathcal{S}_{\mathbf{i}}(\lambda)$, as long as 

\[f_{i_1}(f_{\underline{x}}b_{\lambda})\neq 0,\]

we have 

\[f_{i_{i}}(\underline{x}) = (x_{i_1}+1,\dots , x_{i_{N}}).\]

Now, let $s \in [1,n]$ and let $\mathbf{j} = (j_1,\dots j_N)$ be a reduced expression of $w_0$ such that $j_1 = s$. Let $\Phi^{\mathbf{j}}_{\mathbf{i}}: \mathcal{S}_{\mathbf{i}} \rightarrow \mathcal{S}_{\mathbf{j}}$ be the piecewise linear bijection described explicitly in \cite{conescrystalspatterns}. Then 

\[f_{s}(\underline{x}):= (\Phi_{\mathbf{j}}^{\mathbf{i}})^{-1}f_{s}\Phi_{\mathbf{j}}^{\mathbf{i}}(\underline{x}). \]

We now follow mainly \cite{beaglebvolker}, where an explicit formula for effectively computing the crystal structure on the lattice points of $\mathcal{S}_{\mathbf{i}}(\lambda)$ in terms of rigorous paths is constructed. We will use this construction in the proof of Theorem \ref{thm:epsilon and phi after adding wt zero}.

\begin{definition}
Given ${\bf i}$, $a\in [n-1]$ and $\gamma\in \Gamma_a$ a rigorous path it is completely determined by the set of vertices of $\mathcal P_{\bf i}$ where $\gamma$ changes from one wire to another. These vertices are called \emph{turning points of $\gamma$} and the set of all turning points is denoted by $T_\gamma$.
\end{definition}

\begin{definition}\cite[Definition 4.11]{beaglebvolker}
    \label{def:rs}
Define maps $r,s: \Gamma_{a} \rightarrow \mathbb{Z}$ as follows: 

\begin{align*}
r(\gamma)_{p,q} &= \begin{cases}
\op{sgn}(q-p) & \text{if } v_{p,q} \in T_{\gamma} \\
0 & \text{otherwise}
\end{cases}\\
s(\gamma)_{p,q} &= \begin{cases}
1  & \text{  if } v_{p,q} \in \gamma, p \leq a < q \text{ or } q \leq a < p\\
-1 & \text{ if } v_{p,q} \in \gamma / T_{\gamma}, a < p,q \text{ or } p,q < a.\\
0 & \text{otherwise}
\end{cases}
\end{align*}
\end{definition}

\begin{theorem} \cite[Theorem 5.1]{beaglebvolker} [De Laporte--Genz--Koshevoy]
\label{thm:crystalstructure}
Let $\lambda \in X^{+}, x \in \mathcal{S}_{\mathbf{i}}(\lambda),$ and $a \in [n-1]$. Then the crystal structure in $\mathcal{S}_{\mathbf{i}}(\lambda)$ is given by 

\begin{align*}
\varepsilon_a(x) = \op{max}\left\{ \langle x, r(\gamma) \rangle | \gamma \in \Gamma_{a} \right\}\\
\op{wt}(x) = \lambda - \sum_{k \in [N]}x_{k} \alpha_{k}\\
f_{a}(x) = \begin{cases}
x + s(\gamma^x) & \text{if } \varphi_{a}(x) >0\\
0 & \text{otherwise}
\end{cases}\\
e_{a}(x) = \begin{cases}
x + s(\gamma_x) & \text{if } \varepsilon_{a}(x) >0\\
0 & \text{otherwise}
\end{cases}
\end{align*}

\noindent where $\gamma^{x} \in \Gamma_{a}$, respectively $\gamma_{x} \in \Gamma_{a}$, is the maximal, respectively minimal $\gamma \in \Gamma_{a}$, such that $\langle x , r(\gamma) \rangle = \varepsilon_{a}(x)$. 
\end{theorem}

\section{Embeddings for any reduced expression}

Let $\theta$ be the highest root. As a partition we have $\theta = (2,\underbrace{1,\dots, 1}_{n-2\hbox{ \small times}})$. For $1 \leq i <j \leq n-1$, define 

\begin{align*}
z_{n-1} :& = f_{n-1} f_{n-2} \cdots f_1 b_{\theta} \\
z_{n-2}:&= f_{n-2}f_{n-1}f_{n-3} \cdots f_1 b_{\theta} \\
z_{n-3}:&= f_{n-3}f_{n-2} f_{n-1} \cdots f_1 b_{\theta}\\
\vdots \\
z_{2}:&= f_2 f_3 \cdots f_{n-1} f_1 b_{\theta}\\
z_{1} :&= f_1 f_2 \cdots f_{n-1} b_{\theta}
\end{align*}

\begin{figure}
    \centering
\[\begin{tikzcd}
	&&& {v_{\theta}} \\
	&& \bullet && \bullet \\
	& \bullet && \bullet && \bullet \\
	\bullet && \bullet && \bullet && \bullet \\
	{z_4} && {z_3} && {z_2} && {z_1}
	\arrow["1"', from=1-4, to=2-3]
	\arrow["4", from=1-4, to=2-5]
	\arrow["2"', from=2-3, to=3-2]
	\arrow["4"', from=2-3, to=3-4]
	\arrow["1", from=2-5, to=3-4]
	\arrow["3", from=2-5, to=3-6]
	\arrow["3", from=3-2, to=4-1]
	\arrow["4", from=3-2, to=4-3]
	\arrow["2"', from=3-4, to=4-3]
	\arrow["3", from=3-4, to=4-5]
	\arrow["1", from=3-6, to=4-5]
	\arrow["2", from=3-6, to=4-7]
	\arrow["4", from=4-1, to=5-1]
	\arrow["3", from=4-3, to=5-3]
	\arrow["2", from=4-5, to=5-5]
	\arrow["1", from=4-7, to=5-7]
\end{tikzcd}\]
    \caption{Half of the associated crystal graph $\mathcal{B}(\theta)$ for $\mathfrak{sl}(5,\mathbb{C})$ and its adjoint representation with highest root $\theta = \omega_{1} + \omega_{4}$. Arrows $a \overset{i}{\longrightarrow} b$  mean $f_{i}(a) = b$. }
    \label{fig:adj rep n=5}
\end{figure}
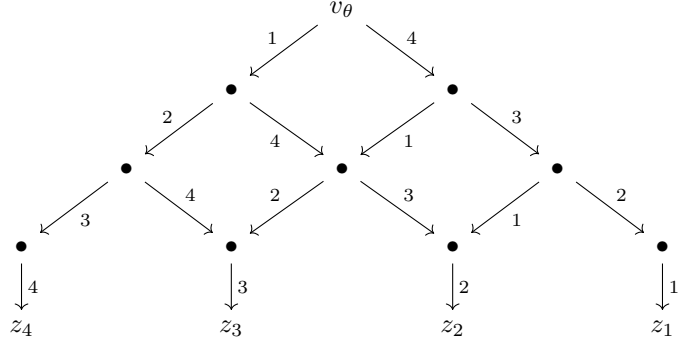


\begin{lemma}
\label{lem:puntoscruz}
Let $\mathcal{B}_{\mathbf{i}}(\theta)$ be the string polytope of highest weight $\theta$ associated to our fixed reduced expression $\mathbf{i}$ of the longest element of $S_n$. 
\begin{itemize}
\item[$i.$] Consider the finite set $\mathcal{B}_{\mathbf{i}}(\theta)_{0} \subset \mathcal{B}_{\mathbf{i}}(\theta)$ of weight zero elements.  We have \[\mathcal{B}_{\mathbf{i}}(\theta)_{0} =  \left \{ z_{1}, \cdots , z_{n-1} \right\}.\] 

\item[$ii.$] For each $1 \leq j \leq n-1$, the coordinates of $z_{n-j}$ in $\mathcal{S}_{\mathbf{i}}$ are given by greedily finding in $\mathbf{i}$ the word $n-j \cdots i_j$ (with $i_j = n-1$ if $j = n-1$, and $i_j = 1$ otherwise) as a subword (up to commutation), reading everything from left to right, placing $1$'s in these coordinates, and zeroes everywhere else. 

\item[$iii.$] Recall the notation from Subsection \ref{diagramnotation}, where we assign to each $v \in \mathcal{B}_{\mathbf{i}}(\theta)$ coordinates $v_{(i,j)}$ depending on the wiring diagram $\operatorname{wd}(\mathbf{i})$. Moreover, for $i<j$:

$$
{z_1}_{(i,j)} = \begin{cases}
1 \mbox{ } \text{if}  & \mbox{ } i = 1 \\
0 \mbox{ }  \mbox{ } & \mbox{ } otherwise.
\end{cases}$$

\noindent and 

$$
{z_{n-1}}_{(i,j)} = \begin{cases}
1 \mbox{ } \text{if}  & \mbox{ } j = n \\
0 \mbox{ }  \mbox{ } & \mbox{ } otherwise.
\end{cases}$$
\end{itemize}
\end{lemma}

\begin{proof}
Note that $ii.$ follows by definition. To prove $i.$, it is enough to show that the elements $z_i$ and $z_j$ are distinct for $i \neq j$, and also that they satisfy the string inequalities. The first condition is automatically implied by the description given in $ii.$ Now we proceed to prove $iii.$ Let us consider \[z_{1} := f_1 f_2 \cdots f_{n-1} b_{\theta}.\]

Consider without loss of generality the wiring diagram $\operatorname{wd}(\mathbf{i})$. By definition, the coordinates of $z_1$ are given by placing a $``1"$ on the vertices of $\operatorname{wd}(\mathbf{i})$ corresponding to the first appearances of a crossing at level $k$, in order from left to right, for $1 \leq k \leq n-1$, and zeroes everywhere else. We claim that these vertices correspond precisely to the crossings of $\ell_1$ with $\ell_j$, for $2\leq j \leq n$. The first crossing at level $1$ clearly corresponds to the crossing of $\ell_1$ and $\ell_{j_1}$, for some $j_1$: if it were not, this would imply that there was a previous crossing at level $1$. We proceed by induction, and assume that the first crossings at level $k < n-1$ correspond to the crossing of $\ell_1$ and $\ell_{j_k}$, for some $j_k$. All of these $j_k$'s are distinct and depend on $\mathbf{i}$, however this follows from the fact that our expression $\mathbf{i}$ is reduced and has nothing to do with the induction hypothesis.  Consider now the crossing at level $k+1$, which is the first such crossing after $v_{1, j_k}$. Because it is the first such crossing, the induction hypothesis implies that $\ell_1$ cannot intersect any other psudoline between $v_{1, j_k}$ and our current crossing. This means that our crossing muust be of the form $v_{1, j_{k+1}}$, for some $j_{k+1}$. This completes the proof. 
\end{proof}

\begin{example}
\label{ex:adjoint}
Let $n= 5$, that is $\mathfrak{g} = \mathfrak{sl}(5,\mathbb{C})$. The highest root in this case is $\theta = \omega_{1} + \omega_{4}$, which is the highest weight of the adjoint representation. See Figure~\ref{fig:adj rep n=5}.

We have chosen to include only the tableaux corresponding to the highest weight vertex and the four weight zero elements $z_{1}, z_{2}, z_{3}, z_{4}$ (numbered as they appear on the picture from right to left). 
Indeed, from Figure~\ref{fig:wiring} we can read off that, for 
${\bf i}=(1,2,3,4,1,2,3,1,2,1)$ we have 
\begin{align*}
z_{1} &= (1,1,1,1,0,0,0,0,0,0) \\
z_{2} &= (0,1,1,1,1,0,0,0,0,0) \\
z_{3} &= (0,0,1,1,0,1,0,1,0,0) \\
z_{4} &= (0,0,0,1,0,0,1,0,1,1).
\end{align*}


\end{example}


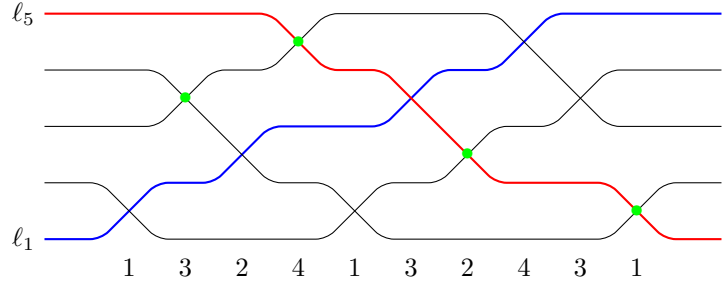
\begin{figure}
    \centering
\begin{tikzpicture}[scale=.75]

\node[left] at (-1,0) {$\ell_1$};
\node[left] at (-1,4) {$\ell_5$};

    \draw[rounded corners,thick, blue] (-1,0)--(0,0)--(1,1)--(2,1)--(3,2)--(5,2)--(6,3)--(7,3)--(8,4)--(10,4)--(11,4);
    \draw[rounded corners] (-1,1)--(0,1)--(1,0)--(4,0)--(5,1)--(6,1)--(7,2)--(8,2)--(9,3)--(10,3)--(11,3);
    \draw[rounded corners] (-1,2)--(1,2)--(2,3)--(3,3)--(4,4)--(7,4)--(9,2)--(10,2)--(11,2);
    \draw[rounded corners] (-1,3)--(1,3)--(3,1)--(4,1)--(5,0)--(9,0)--(10,1)--(11,1);
    \draw[rounded corners,thick,red] (-1,4)--(3,4)--(4,3)--(5,3)--(7,1)--(9,1)--(10,0)--(11,0);
    
    \node at (0.5,-.5){$1$};
    \node at (1.5,-.5){$3$};
    \node at (2.5,-.5){$2$};
    \node at (3.5,-.5){$4$};
    \node at (4.5,-.5){$1$};
    \node at (5.5,-.5){$3$};
    \node at (6.5,-.5){$2$};
    \node at (7.5,-.5){$4$};
    \node at (8.5,-.5){$3$};
    \node at (9.5,-.5){$1$};

    \node at (1.5,2.5) {\color{green} $\bullet$};
    \node at (3.5,3.5) {\color{green} $\bullet$};
    \node at (6.5,1.5) {\color{green} $\bullet$};
    \node at (9.5,.5) {\color{green} $\bullet$};
    
\end{tikzpicture}
    \caption{
    Notice how the subword of ${\bf i}=(1,3,2,4,1,3,2,4,3,1)$ corresponding to $z_1=f_1f_2f_3f_4.b_{\theta}$ (resp. $z_4=f_4f_3f_2f_1.b_{\theta}$) picks up all crossings along the wire $\ell_1$ (resp. $\ell_5$). This is not true for other $z_i$, $i\not\in\{1,4\}$, as for example, $z_3=f_3f_4f_2f_1.b_\theta$ is parametrized by the crossings marked by ${\color{green}\bullet}$ which are spread over different wires.
    }
    \label{fig:example z as paths}
\end{figure}

\begin{theorem}\label{thm:epsilon and phi after adding wt zero}
    Let  $x\in \mathcal S_{\mathbf{i}}\cap \mathbb Z^N$.
    
    \begin{enumerate}
    \item \label{1} If $\varepsilon_a(x)>0$ then $\varepsilon_a(x+z_i)>0$ for $a\in [n-1]$ and $i\in \{1,n-1\}$.
    \item \label{2} If $\varphi_a(x)>0$ then $\varphi_a(x+z_i)>0$ for $a\in [n-1]$ and $i\in \{1,n-1\}.$
    \end{enumerate}
\end{theorem}

\begin{proof}
    For (\ref{1}) we distinguish three cases and make repeated use of Lemma~\ref{lem:puntoscruz}.
    \begin{itemize}
    \item[$a=1=i$] Consider the set of turning points $T_\gamma$. We find some $v_{1,k}\in T_{\gamma}$ for some $k>1$ and no other turning point along $\ell_1$. Hence, $\langle r(\gamma),z_i\rangle=1$ so that
    \[
    \varepsilon_a(x+z_i)=\max \{\langle x,r(\gamma)\rangle +1\}_{\gamma\in \Gamma_a} =\varepsilon_a(x)+1.
    \]
    \item[$a=n-1=i$] Similar as in the first case, each $T_\gamma$ contains a vertex $v_{k,n}$ for some $k<n$ and no other turning point along $\ell_n$. The rest follows as above.
    \item[$a\not=i$] We claim that $\langle r(\gamma),z_i\rangle =0$ for all $\gamma\in\Gamma_a$. Each path $\gamma$ has two possibilities: either it travels along $l_i$, here $i=1,n$, (not only intersect it, but contains a segment of $\ell_i$) at some point or it does not. In the latter case no $v_{i,k}$ (or $v_{k,i}$) is in $T_\gamma$, so the claim holds.
    In the former case there is some $v_{k,i}$ (respectively $v_{i,k}$) in $T_{\gamma}$ but as $a\not=i$ there is also another vertex of form $v_{i,k'}$ (respectively $v_{k',i}$) in $T_{\gamma}$. As $i\in \{1,n\}$ the two vertices contribute to $r(\gamma)$ with opposite signs, hence the claim follows. In particular,
    \[
    \varepsilon_a(x)=\varepsilon_a(x+z_i).
    \]
    \end{itemize}
    To show (\ref{2}) it is enough to note that 
$\varphi_{i}(x) = \langle \operatorname{wt}(x), \alpha_{i} \rangle + \varepsilon_{i}(x).$ 
\end{proof}

\begin{remark}
    It is possible to write a proof of Theorem \ref{thm:epsilon and phi after adding wt zero}
using only the basic definition of the crystal structure on $\mathcal{S}_{\mathbf{i}}(\lambda)$ given in \ref{sec:crystalstructure}, without Definition \ref{thm:crystalstructure} from \cite{beaglebvolker}. We thank Bea De Laporte for this observation. However we found using her joint results from \cite{beaglebvolker} more illustrative. 
\end{remark}

\begin{corollary}
\label{thm:crystal1}
The map $x \mapsto x + z_{i}$, for $i \in \{1, n-1 \}$ induces an injective crystal morphism

\[\mathcal{S}_{\bf i}(\lambda) \overset{\psi^{n}_{i}}{\longrightarrow} \mathcal{S}_{\bf i}(\lambda + \theta).\]
\end{corollary}

\begin{proof}
Note that if $x \in \mathcal{S}_{\mathbf{i}} \cap \mathbb{Z}^{N}$, then since $\mathcal{S}_{\mathbf{i}}$ is a convex cone, then $x+z_i \in \mathcal{S}_{i} \cap \mathbb{Z}^{N}$ as well. The rest follows from Theorem \ref{thm:epsilon and phi after adding wt zero}.
\end{proof}

We define two restriction maps from $\mathfrak{sl}_n$ to $\mathfrak{sl}_{n-1}$
that induce projections of the corresponding string cones. 
Combinatorially it may be described by \emph{removing} the first respectively the $n^{\text{th}}$ wire, denoted by 
\[
\text{pr}_1:\mathbb Z^N\to \mathbb  Z^{N-n+1}, \quad \text{resp.} \quad \text{pr}_n:\mathbb Z^N\to \mathbb  Z^{N-n+1}
\]
For example, fixing ${\bf i}=(1,2,\cdots,{n-1},1,\cdots,{n-2},\cdots, 1,2,1)$
 the word corresponding to $\text{pr}_1$ is  $\text{pr}_1({\bf i})=(1,2,\cdots,{n-2},1,\cdots ,{n-3},\cdots ,1,2,1)$ while the word corresponding to $\text{pr}_n$ is $\text{pr}_n({\bf i})=(2,\cdots, {n-1},2,\cdots, {n-3},\cdots ,2,3,2)$.
Hence, at the level of string cones we obtain 
\begin{eqnarray*}
\text{pr}_1(x_1,\dots,x_N)&=&(x_n,\dots,x_N), \quad \text{resp.}\\
\text{pr}_n(x_1,\dots,x_N)&=&(x_1,\dots,x_{n-2},x_{n},\dots,x_{2n-3},\dots,x_{N-2})
\end{eqnarray*}
Notice that we have precisely $\text{pr}_1(z_1)=0=\text{pr}_n(z_{n-1})$.
To distinguish the reference Lie algebra in the following result we denote $\varepsilon_a=\varepsilon_a^n$ and $\varphi_a=\varphi_a^n$ to emphazise that we are working with $\mathfrak{sl}_n$.


\begin{corollary}
\label{cor:proj}
    Fix a reduced expression ${\bf i}$ and $x\in \mathcal S_{\bf i}\cap \mathbb Z^N$. 
    Then for every $1\le i\le n-1$ there exists a sequence of projections $\textbf{pr}:=\text{pr}_{j_k}\circ\dots \circ \text{pr}_{j_1}$ with $j_i\in \{1,n-i\}$ such that
    \begin{enumerate}
        \item If $\varepsilon_a^{n-k}(\textbf{pr}(x))>0$ then $\varepsilon^{n-k}_a(\textbf{pr}(x+z_i))>0$, and
        \item If $\varphi^{n-k}_a(\textbf{pr}(x))>0$ then $\varphi^{n-k}_a(\textbf{pr}(x+z_i))>0$.
    \end{enumerate}
\end{corollary}
\begin{proof}
    From the crystal graph of $V(\theta)$ for $\mathfrak{sl}_n$ we see that after sufficiently many, say $k$ many, projections $z_i=z_i^n$ has image $z^{n-k}_1$ or $z^{n-k}_n$. Now Theorem~\ref{thm:epsilon and phi after adding wt zero} applies. 
\end{proof}

\section{Outlook: Atomic decompositions}

In \cite{Pat}, a geometrical approach to the problem of computing $q-$analogues of weight multiplicities $K_{\lambda,\mu}(q)$, also known as Kostka--Foulkes polynomials, was proposed. This approach reduces the problem to finding so-called \textit{atomic decompositions} of crystals, together with compatible \textit{swapping functions}. In \cite{PT23}, this approach was carried out in type $C_2$. The atomic decompositions obtained in \cite{PT23} were obtained using certain embeddings of crystals in terms of string polytopes, which, to the authors' knowledge, had not been considered before. Such embeddings are the main motivation for this work. In fact, even in type $A_{n-1}$, only one atomic decomposition of crystals is known \cite{LL21,Pat}. In Conjecture \ref{conj:atoms}, which is formulated as the result of numerous computer calculations using SageMath \cite{sagemath}, we conjecture a way to systematically obtain $n$ distinct atomic decompositions of the type $A_n$ crystals $\mathcal{B}(k\theta)$, where we recall that $\theta$ is the highest root.  

\begin{definition} Given a dominant integral weight $\lambda$, consider the multiset of weights $\op{wt}(b)$ for $b \in \mathcal{B}(\lambda)$ and denote it by $\mathcal{N}(\lambda)$. Additionally, let $N(\lambda)$ the associated set, i.e. the set obtained from $\mathcal{N}(\lambda)$ by removing multiple occurrences of a given weight. More generally, for a subset $\mathcal{A} \subset \mathcal{B}(\lambda)$, we denote by $\mathcal{N}(\mathcal{A})$ the sorresponding multiset of weights $\op{wt}(a)$ for $a \in \mathcal{A}$.
\end{definition}

\begin{definition}
An \textit{atom} in $\mathcal{B} (\lambda)$  is a subset $\mathcal{A} \subset \mathcal{B} (\lambda)$ such that the multiset $\mathcal{N}(\mathcal{A})$ coincides with the set $N(\mu)$ for some dominant integral weight $\mu \leq \lambda$. We say that the atom is a \textit{big atom} if $\mu = \lambda$.
\end{definition}
                  
\begin{conjecture}  
\label{conj:atoms}
Let $i \in [1,n-1]$. Then the subset 

\[\mathcal{A}_{i} := (\bigcup_{j \neq i}\text{im}(\psi^{n}_{j}))^{c}\]

\noindent is a big atom in $\mathcal{B}(k\theta)$.
\end{conjecture}

 \begin{example}
 \label{ex:adjoint2}
 Let us go back to Example \ref{ex:adjoint}. We see that if $\mathcal{B}_{triv} = \{ \cdot\}$ denotes the trivial crystal for $\mathfrak{sl}(5,\mathbb{C})$, then $\psi^{5}_{4}(\cdot) = z_4$. 
Consider $\mathcal{P}(\theta) : = \op{im}(\psi^{5}_{4})^{c}$. Then $\op{res}^{5}_{4}(\mathcal{B}(\theta)) \cap \mathcal{P}(\theta)$ contains a crystal isomorphic to the crystal of the adjoint representation for $\mathfrak{sl}(4,\mathbb{C})$, with highest weight vertex $f_{4}(v_{\theta})$, where $v_{\theta} \in \mathcal{B}(\theta)$ is the highest weight vertex. Now we remove $\op{im}(\psi^{5}_{3})$ from $\mathcal{P}(\theta)$ and subsequently $\op{im}(\psi^{5}_{2})$ from what remains, to get our big atom $\mathcal{A}(\theta)$, which has weight zero element corresponding to $z_1$.
\end{example}

\begin{example}
\label{ex:sl3}
This is a continuation of Example \ref{ex:adjoint}. Consider $\mathfrak{g} = \mathfrak{sl}(3,\mathbb{C})$ and $\lambda = \lambda_{1}\omega_{1} + \lambda_{2}\omega_{2}$. Note that if either  $\lambda_1 = 0$ or $\lambda_2 = 0$ then $\mathcal{B}(\lambda)$ is already an atom. Let $\lambda$ be a dominant weight such that $\lambda_{1} \neq 0$ and $\lambda_{2} \neq 0$. Then $\lambda - \theta$ is again a dominant weight.  Define 
\[\mathcal{A}(\lambda):= \psi^{3}_{2}(\mathcal{B}(\lambda -\theta))^{c} \subset \mathcal{B}(\lambda).\]

We claim that $\mathcal{A}(\lambda)$ coincides with the \textit{big atom} $\mathcal{A}^{pat}(\lambda)$ 
 defined by Patimo \cite{Pat}, and, independently, by Lenart--Lecouvey \cite{LL21}. We use the definition of $\mathcal{A}^{pat}(\lambda)$ as the $f_2$ and $W$-orbit of the highest weight element. We refer the reader to \cite{BSch17} for the definition of the action of the Weyl group on a crystal (the simple reflection $s_i$ acts by reversing the corresponding $i$-string). To show this, we will work on the string polytope $\mathcal{S}_{\mathbf{i}}(\lambda)$, where $\mathbf{i} = (2,1,2)$. The map $\psi^{3}_{2}$ is defined by 
\begin{align*}
x = (x_{1},x_{2},x_{3}) \mapsto (x_{1},x_{2}+1,x_{3}+1).
\end{align*}

The defining inequalities for $\mathcal{S}_{\mathbf{i}}(\lambda)$ are:
\begin{align*}
0 \leq x_{3} & \leq \lambda_{2}\\
0 \leq x_3 &\leq x_{2}  \leq \lambda_{1} + x_{3}\\
0 \leq x_{1} & \leq \lambda_{2} -2x_{3}+x_{2}.
\end{align*}

Whenever $x_2 > 0$ and $x_{3} > 0$ in $\mathcal{A}(\lambda)$, the elements in $\mathcal{A}(\lambda)$ must additionally  also satisfy $x_2 > \lambda_1 + x_3 - 1$, therefore $x_2 - x_3 = \lambda_1 $. 
Note that if  we take an element $x = (x_1, x_2,x_3)$ with fixed weight $\op{wt}(x) = \mu$, then $x_1$ is uniquely determined by the values of $x_2$ and $x_3$. This means that necessarily $\mathcal{A}(\lambda) \cap \mathcal{S}_{\mathbf{i}}(\lambda)_{\mu}$ consists of at most one element. Now, the elements $x \in \mathcal{S}_{\mathbf{i}}(\lambda)$ such that either $x_2 = 0$ or $x_3 =0$ all belong to $\mathcal{A}(\lambda)$. If both are zero, then $x_1$ determines $x$ uniquely since its weight is $\lambda - x_1 \alpha_2$. If $x_3 = 0$ and $x_2 >0$, or if $x_3 >0$ and $x_2 =0$, then $x$ is similarly uniquely determined by its weight. To prove that $\mathcal{A}^{pat}(\lambda) = \mathcal{A}(\lambda)$, it is now enough to show that $\mathcal{A}(\lambda) \supset \mathcal{A}^{pat}(\lambda)$. This follows from the following claim:

\begin{align*}
\hbox{The set }\mathcal{A}(\lambda) \hbox{ is stable under the action of the Weyl group and the action of } f_2.
\end{align*}

The first part of the statement follows since the image $\psi^{3}_{2}(\mathcal{B}(\lambda - \theta))$ is stable under the action of the Weyl group by Corollary \ref{thm:crystal1}. It follows directly from the definitions that it is not possible to have $x \in \mathcal{A}(\lambda)$ such that $f_2(x) = (x_1 +1, x_2, x_3) \in \psi^{3}_{2}(\mathcal{B}(\lambda - \theta))$.

\end{example}

\begin{remark}
Consider the model for the crystal $\mathcal{B}(\lambda)$ given by semi-standard Young tableaux of shape $\lambda$. Then it is not difficult to show, using the crystal structure on semi-standard Young tableaux, that the map 

\[\phi^{n}_{n-1}: \mathcal{B}(\lambda) \rightarrow \mathcal{B}(\lambda + \theta)\]

\noindent is given by adding to $T \in \mathcal{B}(\lambda)$ the column of length $n-1$ and entries $1,2,\dots, n-1$ as the left-most column and a box with entry $n$ at the end of the first row. Note that if $\op{ch}(-)$ denotes the Lascoux--Sch\"utzenberger charge (cf./ \cite[Definition 2.4.22]{butler1994subgroup}) or \cite{lascoux1978conjecture}, then from this one sees immediately that for any $T \in \mathcal{B}(\lambda)$:

\[\op{ch}(\phi^{n}_{n-1}(T)) = \op{ch}(T)+1. \]
\end{remark}

\bibliographystyle{abbrv}
\bibliography{bibliography}

\end{document}